\documentclass[10pt]{amsart}

\usepackage[english]{babel}
\usepackage{amsmath}
\usepackage{amsthm}
\usepackage{enumerate}
\usepackage[percent]{overpic}
\usepackage{graphicx}
\usepackage[top=2.5cm, bottom=2.5cm, left=3cm, right=3cm]{geometry}
\newtheorem{thm}{Theorem}
\newtheorem{lem}[thm]{Lemma}
\newtheorem{prop}[thm]{Proposition}
\newtheorem{cor}[thm]{Corollary}

\newtheorem{claim}{Claim}

\theoremstyle{remark}
\newtheorem{rem}[thm]{Remark}
\newtheorem{exam}[thm]{Example}

\makeatletter
\newlength{\vbraceheight}
\def\vbig#1{{\resizebox{!}{\vbraceheight}{$\left#1\vbox to\vbraceheight{}\right.\n@space$}}}
\def\vbigl{\mathopen\vbig}
\def\vbigr{\mathclose\vbig}
\makeatother

\numberwithin{equation}{section}

\newcommand{\spin}{\ifmmode{\rm Spin}\else{${\rm spin}$\ }\fi}
\newcommand{\spinc}{\ifmmode{{\rm Spin}^c}\else{${\rm spin}^c$}\fi}

\newcommand{\Z}{\mathbb{Z}}
\newcommand{\Q}{\mathbb{Q}}
\newcommand{\vol}{\mathrm{vol}}

\DeclareMathOperator{\lk}{link}

\title{Characterizing slopes for the $(-2,3,7)$-pretzel knot}
\author{Duncan McCoy}
\address {Université du Québec à Montréal}
\email{mc\_coy.duncan@uqam.ca}
\date{}
\begin{document}

\begin{abstract}
In this note we exhibit concrete examples of characterizing slopes for the knot $12n242$, aka the $(-2,3,7)$-pretzel knot. Although it was shown by Lackenby that every knot admits infinitely many characterizing slopes, the non-constructive nature of the proof means that there are very few hyperbolic knots for which explicit examples of characterizing slopes are known.
\end{abstract}
\maketitle

\section{Introduction}
Given a knot $K\subseteq S^3$, we say that $p/q\in \Q$ is a {\em characterizing slope} for $K$ if the oriented homeomorphism type of the manifold obtained by $p/q$-surgery on $K$ determines $K$ uniquely. That is, $p/q$ is a characterizing slope for $K$ there does not exist any knot $K'\neq K$ such that  $S^3_{p/q}(K)\cong S^3_{p/q}(K')$. It was shown by Lackenby that every knot admits infinitely many characterizing slopes and for a hyperbolic knot any slope $p/q$ with $q$ sufficiently large is characterizing \cite{Lackenby2019characterizing}. Although these results show the existence of characterizing slopes, the proofs are non-constructive and so there are very few hyperbolic knots for which explicit examples of characterizing slopes are known. Ozsv\'ath and Szab\'o have shown every slope is characterizing for the figure-eight knot $4_1$ \cite{Ozsvath2006trefoilsurgery} and recent work of Baldwin and Sivek implies that every non-integer slope is characterizing for $5_2$ \cite{Baldwin52}. The aim of this article is to exhibit explicit examples of characterizing slopes for the knot $12n242$, also known as the $(-2,3,7)$-pretzel knot. Since $12n242$ is a hyperbolic $L$-space knot ---Fintushel and Stern showed that it admits two lens space surgeries \cite{Fintushel1980lenssurgery}---  it has only finitely many non-characterizing slopes that are not negative integers \cite{McCoy2018hyp_char}. The following theorem is a quantitative version of this fact. As far as the author is aware, these are the first known explicit examples of characterizing slopes on a hyperbolic knot with genus greater than one.

\begin{thm}\label{thm:12n242}
Any slope $p/q$ satisfying at least one of the following conditions is a characterizing slope for $12n242$:
\begin{enumerate}[(i)]
\item $q\geq 49$;
\item $p\geq \max\{24q, 441\}$; or
\item $q\geq 2$ and $p\leq - \max\{12+4q^2-2q,441 \}$.
\end{enumerate}
\end{thm}
The key input allowing us to prove Theorem~\ref{thm:12n242} is the fact that $12n242$ is one of the knots with smallest volume (up to reflection it one of only three hyperbolic knots with volume smaller than 3.07) \cite{Gabai_low_vol}. A result of Futer, Kalfagianni and Purcell on the change in volume of a hyperbolic manifold under Dehn filling \cite{Futer2008Dehn_filling} can then be used to restrict potential non-characterizing slopes coming from surgeries on hyperbolic knots with large volume (and satellites thereof). A miscellany of invariants can then be used to rule out non-characterizing slopes coming from hyperbolic knots with small volume.
\begin{figure}[ht]
  \centering
   \includegraphics[width=0.4\linewidth]{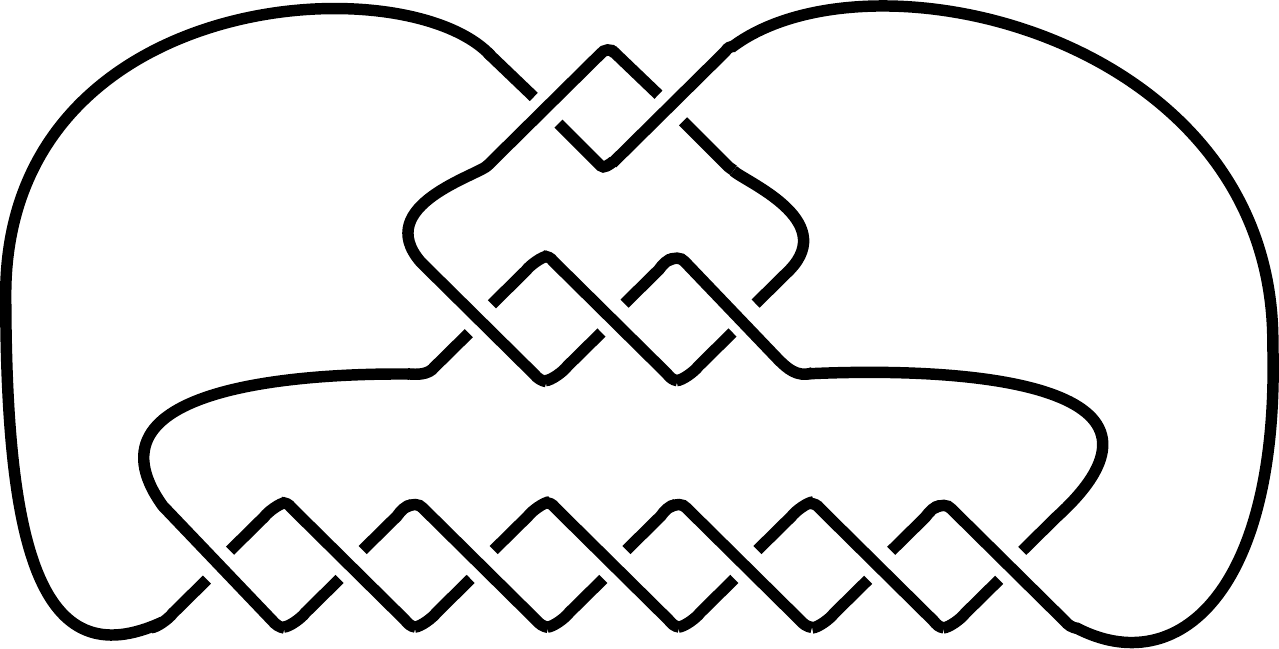}
 \caption{The main protagonist: $12n242$.}
 \label{fig:nugatory}
\end{figure}

In principle, one could use a similar approach to derive information about the characterizing slopes of the other small volume knots: $4_1$ and $5_2$. However, much better results have already been obtained by other means for both of these knots \cite{Ozsvath2006trefoilsurgery, Baldwin52}, so we restrict our analysis to $12n242$.

We note that Theorem~\ref{thm:12n242} says nothing about negative integer characterizing slope. Although there are knots which possess infinitely many integer non-characterizing slopes \cite{Baker2018Noncharacterizing}, all known examples admit infinitely many non-characterizing slopes of both sign. This suggests that $12n242$ (and $L$-space knots more generally) should admit only finitely many integer non-characterizing slopes. However establishing such a result remains an interesting and challenging problem.

\subsection*{Non-characterizing slopes}
Lackenby has shown for a hyperbolic knot $K$ any slope $p/q$ with $q$ sufficiently large is characterizing for $K$ \cite{Lackenby2019characterizing}. For example, Theorem~\ref{thm:12n242} shows that $q\geq 49$ is sufficiently large for $12n242$. However the ``sufficiently large'' here is inherently dependant on the specific knot in question. To illustrate this dependence, we exhibit a family of hyperbolic two-bridge knots $\{K_q\}_{q\geq 1}$ such that for each $q$, the slope $\frac{1}{q}$ is non-characterizing for $K_q$. This family is shown in Figure~\ref{fig:2bridge_exam} with the details of the construction discussed in Section~\ref{sec:construction}.

\subsection*{Conventions}
The following notational conventions will be in force throughout the paper:
\begin{itemize}
\item When considering a rational number $p/q\in \Q$, we will always assume this to be written with $p$ and $q$ coprime and $q\geq 1$.
\item Given two oriented 3-manifolds $Y$ and $Y'$, we will use $Y\cong Y'$ to denote the existence of an orientation-preserving homeomorphism between them.
\item For a knot $K$, we will denote its {\em Alexander polynomial} by $\Delta_{K}(t)$. We will always assume this is normalized so that $\Delta_{K}(1)=1$ and $\Delta_K(t)=\Delta_K(t^{-1})$.
\item Given a knot $K$ in $S^3$, we will use $mK$ to denote its mirror.
\item An $L$-space knot is one which admits positive $L$-space surgeries.
\end{itemize}
\subsection*{Acknowledgements} The author would like to thank Steve Boyer and Patricia Sorya for interesting conversations. He would also like to acknowledge the support of NSERC and FRQNT.
\section{Preliminaries}
In this section we gather together all the auxiliary results required for the proof of Theorem~\ref{thm:12n242}.

\subsection{Knots of small volume}
First we use the fact that Gabai, Haraway, Meyerhoff, Thurston and Yarmola have classified the hyperbolic 3-manifolds of small volume \cite{Gabai_low_vol}.
\begin{thm}\label{thm:low_vol_knots}
If $K$ is a hyperbolic knot in $S^3$ with $\vol(K)\leq 3.07$, then 
\[K\in \{4_1, 5_2, 12n242, m5_2, m12n242\}.\]
\end{thm}
\begin{proof}
Gabai, Haraway, Meyerhoff, Thurston and Yarmola have shown that there are exactly 14 one-cusped orientable hyperbolic 3-manifolds with hyperbolic volume less than or equal to 3.07 and that these are $\mathtt{m003}$, $\mathtt{m004}$, $\mathtt{m006}$, $\mathtt{m007}$, $\mathtt{m009}$, $\mathtt{m010}$, $\mathtt{m011}$, $\mathtt{m015}$, $\mathtt{m016}$, $\mathtt{m017}$, $\mathtt{m019}$, $\mathtt{m022}$, $\mathtt{m023}$ and $\mathtt{m026}$ \cite[Theorem~1.5]{Gabai_low_vol}.
Precisely three of these arise as the complements of knots in $S^3$: $\mathtt{m004}$, $\mathtt{m015}$ and $\mathtt{m016}$ are (ignoring orientations) the complements of $4_1$, $5_2$ and $12n242$, respectively.
\end{proof}
We will informally refer to the five knots in Theorem~\ref{thm:low_vol_knots} as the ``low volume knots'' and the remaining hyperbolic knots as the ``large volume knots''. For our purposes it will be useful to note that the volume of $4_1$ satisfies
\begin{equation}\label{eq:4_1vol}
\vol(4_1)\approx 2.0988\leq 2.1
\end{equation}
and the volume of $12n242$ satisfies
\begin{equation}\label{12n242_vol}
2.82\leq \vol(12n242)\approx 2.821 \leq 2.83.
\end{equation}

\subsection{Slope lengths}
Let $K$ be a hyperbolic knot in $S^3$, that is, $S^3\setminus K$ admits a complete finite-volume hyperbolic structure with one cusp. Given a slope $\sigma$ on $K$ and horoball neighbourhood $N$ of the cusp we can assign a length to $\sigma$ by considering the minimal length of a curve representing $\sigma$ on $\partial N$ (measured in the natural Euclidean metric on $\partial N$). Since $S^3\setminus K$ has a unique cusp, there is a unique maximal horoball neighbourhood of this cusp. We will use $\ell_K(\sigma)$ to denote the length of $\sigma$ with respect this maximal horoball neighbourhood.

\begin{lem}\label{lem:vol_length_bound}
Let $K$ and $K'$ be hyperbolic knots in $S^3$ with $\vol(K')<\vol(K)$. If $r$ and $r'$ are slopes such that $S_{r}^3(K)\cong S_{r'}^3(K')$, then
\[
\ell_{K}(r)<\frac{2\pi}{\sqrt{1-\left(\frac{\vol(K')}{\vol(K)}\right)^{\frac23}}}
\]
\end{lem}
\begin{proof}
Futer, Kalfagianni and Purcell have shown that if $\ell=\ell_{K}(r)>2\pi$, then we have the following volume bound \cite[Theorem~1.1]{Futer2008Dehn_filling}:
\[
\vol(K)\left(1- \left(\frac{2\pi}{\ell}\right)^2\right)^{\frac32}< \vol(S_{r}^3(K)).
\]
Furthermore, since Thurston showed that volume strictly decreases under Dehn filling \cite{Thurston1982KleinianGroups}, we have that $\vol(S_{r}^3(K))=\vol(S_{r'}^3(K'))<\vol(K')$. Together these bounds give
\[
\vol(K)\left(1- \left(\frac{2\pi}{\ell}\right)^2\right)^{\frac32}<\vol(K'),
\]
which can be easily rearranged to give the desired bound on $\ell_K(p/q)$.
\end{proof}
Next we need a mechanism for converting bound on $\ell_K(p/q)$ into bounds on $p$ and $q$.
\begin{lem}\label{lem:bounding_pq}
Let $K\subseteq S^3$ be a hyperbolic knot of genus $g(K)$. Then 
\begin{enumerate}[(a)]
\item\label{it:q_bound} $|q|\leq 1.79 \ell_K(p/q)$ and
\item\label{it:p_bound} $|p|\leq 1.79\ell_K(p/q)(2g(K)-1)$
\end{enumerate}
\end{lem}
\begin{proof}
Let $N$ be a horocusp neighbourhood in the knot complement of $S^3_K$. Let $A$ be the area of $\partial N$ (equipped with its Euclidean metric)
A simple geometric argument (e.g. as used by Cooper and Lackenby \cite[Lemma~2.1]{Cooper1998Dehn}) shows that for any two slopes of $K$ we have 
\[\ell_K(\alpha)\ell_K(\beta)\geq A\Delta(\alpha,\beta),\]
where $\Delta(\alpha,\beta)$ denotes the distance between $\alpha$ and $\beta$. (cf. \cite[Lemma~8.1]{Agol2000BoundsI}). Since Cao and Meyerhoff have shown there always exists a horocusp neighbourhood $N$ with $\mathrm{Area}(\partial N)\geq 3.35$ \cite{Cao2001cusped}, this establishes the bound 
\[\ell_K(\alpha)\ell_K(\beta)\geq 3.35\Delta(\alpha,\beta),\]
for all slopes $\alpha$ and $\beta$.
 Since $\Delta(1/0,p/q)=|q|$ and $\ell_K(1/0)\leq 6$ by the 6-theorem of Agol and Lackenby \cite{Agol2000BoundsI, Lackenby2003Exceptional}, this gives the bound \eqref{it:q_bound}. Since $\Delta(0/1,p/q)=|p|$ and $\ell_K(0/1)\leq 6(2g-1)$ by \cite[Theorem~5.1]{Agol2000BoundsI}, this also gives the bound \eqref{it:p_bound}.
\end{proof}

\subsection{Hyperbolic surgeries on satellite knots}
We will use the following result to understand non-characterizing slopes coming from satellite knots.
\begin{lem}\label{lem:satellite}
Let $K'$ be a satellite knot such that $S_{p/q}^3(K')$ hyperbolic for some $p/q\in \Q$. Then there is a hyperbolic knot $J$ with $g(J)<g(K')$ and an integer $w>1$ such that $S_{p/q}^3(K)\cong S_{p/(qw^2)}^3(J)$. Moreover, if $q\geq 2$, then $K'$ is a cable of $J$ with winding number $w$.
\end{lem}
\begin{proof}
Let $T$ be an incompressible torus in $S^3\setminus K'$. We can consider $K'$ as a knot in the solid torus $V$ bounded by $T$. Thus we can consider $K'$ as a satellite with companion given by the core $J$ of $V$. By choosing $T$ to be ``innermost'', we can ensure that $S^3\setminus J$ contains no further incompressible tori. That is, we can assume that $J$ is not a satellite knot. By Thurston's trichotomy for knots in $S^3$, this implies that $J$ is a torus knot or a hyperbolic knot \cite{Thurston1982KleinianGroups}. Since $S_{p/q}^3(K')$ is hyperbolic, it is atoroidal and irreducible. Consequently, after surgery the solid torus $V$ must become another solid torus. However, Gabai has classified knots in a solid torus with non-trivial solid torus surgeries, showing that $K'$ is either a torus knot or a 1-bridge braid in $V$ \cite{Gabai1989solidtori}. Moreover since solid torus fillings on 1-bridge braids only occur for integer surgery slopes, $K'$ is a cable of $J$ unless $q=1$. In either event, we have that
\[ S_{p/q}^3(K') \cong S_{p/q'}^3(J),\]
where the slope $p/q'$ is determined by the curve bounding a disk after surgering $V$. Using a homological argument one can show that $q'=qw^2$, where $w>1$ is the winding number of $K'$ in $V$ \cite[Lemma~3.3]{Gordon1983Satellite}.
Since $S_{p/q}^3(K)$ is a hyperbolic manifold, $J$ cannot be a torus knot. It follows that $J$ must be a hyperbolic knot. The only remaining statement is the inequality $g(J)<g(K')$. This follows from Schubert's formula for the genus of a satellite knot \cite{Schubert}, which
asserts that for a knot $K'=P(J)$ with pattern $P$ of winding number $w\geq 0$, there is a constant $g(P)\geq 0$ such that
\[
g(K')=g(P)+wg(J).
\]
We obtain the necessary inequality since $w\geq 2$.
\end{proof}

\subsection{The Casson-Walker invariant.}
It will also be convenient to use surgery obstructions derived from the Casson-Walker invariant \cite{Walker1992extension}. For any rational homology sphere $Y$, this is a rational-valued invariant $\lambda(Y)\in\Q$. Boyer and Lines showed that this satisfies the following surgery formula \cite{Boyer1990Surgery}:
\begin{equation*}\label{eq:BoyerLines}
\lambda(S_{p/q}^3(K'))=\lambda(S^3_{p/q}(U)) + \frac{q}{2p} \Delta_{K'}''(1),
\end{equation*}
where $\Delta_K''(1)$ denotes the second derivative of the Alexander polynomial $\Delta_K(t)$ evaluated at $t=1$. This formula immediately yields the following observation.
\begin{lem}\label{lem:CW_obstruction}
Let $K$ and $K'$ be knots. If there is a non-zero $p/q\in\Q$ such that $S_{p/q}^3(K)\cong S_{p/q}^3(K')$, then $\Delta''_{K}(1)=\Delta''_{K'}(1)$.
\end{lem}
Lemma~\ref{lem:CW_obstruction} can be used to obstruct non-characterizing slopes coming from cables.
\begin{lem}\label{lem:CW_satellite_obstruction}
Let $K$ and $K'$ be knots. If there is $K''$ a non trivial cable of $K'$ and a non-zero slope $p/q\in\Q$ such that $S_{p/q}^3(K)\cong S_{p/q}^3(K'')$, then there are coprime integers $r,s$, with $s\geq 2$ such that
\[
\Delta_K''(1)= \frac{(r^2-1)(s^2-1)}{12} + s^2\Delta''_{K'}(1)
\]
\end{lem}
\begin{proof}
Suppose that $K''$ is the $(r,s)$-cable of $K'$, where $s\geq 2$ is the winding number.
By the usual formula for the Alexander polynomial of a satellite knot, we have that
\begin{equation*}\label{eq:satellitealexpoly}
\Delta_{K''}(t)=\Delta_{K'}(t^s)\Delta_{T_{r,s}}(t).
\end{equation*}
where $T_{r,s}$ denotes the $(r,s)$-torus knot. Taking second derivatives we obtain\footnote{The reader should note that since $\Delta_{K}(t)=\Delta_{K}(t^{-1})$  we have that $\Delta'_{K}(1)=0$}
\begin{equation}\label{eq:satellitederiv}
\Delta_{K''}''(1)= \Delta_{T_{r,s}}''(1)+ s^2\Delta_{K'}''(1).
\end{equation}
Since the torus knot $T_{r,s}$ has symmetrized Alexander polynomial 
\[
\Delta_{T_{r,s}}(t)=t^{-\frac{(r-1)(s-1)}{2}}\frac{(t^{rs}-1)(t-1)}{(t^r-1)(t^s-1)},
\]
one can calculate that\footnote{Since the direct calculation is somewhat involved, we include a derivation  for completeness, but relegate it to the appendix.}
\begin{equation}\label{eq:torus_deriv}
\Delta_{T_{r,s}}''(1)=\frac{(r^2-1)(s^2-1)}{12}.
\end{equation}
Combining \eqref{eq:satellitederiv} and \eqref{eq:torus_deriv} with Lemma~\ref{lem:CW_obstruction} gives the desired statement.
\end{proof}
We will be applying these obstructions to the knots $5_2$ and $12n242$.
These have symmetrized Alexander polynomials:
\begin{align*}
\Delta_{5_2}(t)&=2t^{-1}-3+2t\\
\Delta_{12n242}(t)&=t^{-5}-t^{-4}+t^{-2}-t^{-1}+1-t +t^2-t^4+t^5.
\end{align*}
Hence one finds that 
\begin{equation}\label{eq:delta_examples}
\Delta_{5_2}''(1)=4 \quad\text{and}\quad \Delta_{12n242}''(1)=24.
\end{equation}

\subsection{An obstruction from $\nu^+$}
Here we take some input from knot Floer homology. Recall that for a knot $K$ in $S^3$, Ni and Wu derived a non-increasing sequence of non-negative integers $V_0(K),V_1(K), \dots$ from the knot Floer chain complex which can be used to calculate the $d$-invariants of surgeries on $K$. For $p/q>0$ and an appropriate identification of $\spinc(S_{p/q}^3(K))$ and $\spinc(S_{p/q}^3(U))$ with $\{0,1, \dots, p-1\}$, we have \cite[Proposition~1.6]{ni2010cosmetic}
\begin{equation}\label{eq:NiWuformula}
d(S_{p/q}^3(K),i)=d(S_{p/q}^3(U),i)-2\max\left\{V_{\lfloor\frac{i}{q}\rfloor}(K),V_{\lceil\frac{p-i}{q}\rceil}(K)\right\}.
\end{equation}
Hom and Wu defined the invariant $\nu^+(K)$ to be the smallest index $i$ for which $V_i=0$ \cite{Hom2016refinement}. In particular we have $\nu^+(K)=0$ if and only if $V_0=0$.

\begin{lem}\label{lem:nu+application}
Let $K$ be a knot such that $\nu^+(K)>0$ and $\nu^+(mK)=0$. Then there is no non-zero slope $p/q\in \Q$ such that $S_{p/q}^3(K)\cong S_{p/q}^3(mK)$.
\end{lem}
\begin{proof}
Since $-S_{p/q}^3(K)\cong S_{-p/q}^3(mK)$, we can assume that $p/q>0$. Summing the formula \eqref{eq:NiWuformula} over all \spinc-structures on $S_{p/q}^3(mK)$ and $S_{p/q}^3(K)$ we see that
\begin{align*}
\sum_{i=0}^{p-1} d(S_{p/q}^3(mK),i) - \sum_{i=0}^{p-1} d(S_{p/q}^3(K),i)&=2 \sum_{i=0}^{p-1} \max\left\{V_{\lfloor\frac{i}{q}\rfloor}(K),V_{\lceil\frac{p-i}{q}\rceil}(K)\right\}\\
&\geq 2V_0>0.
\end{align*}
Which implies that $S_{p/q}^3(mK)$ and $S_{p/q}^3(mK)$ cannot be homeomorphic.
\end{proof}
\begin{rem}\label{rem:nu+ofLspace}
We note that Lemma~\ref{lem:nu+application} applies to any non-trivial $L$-space knot (and in particular $12n242$). For a non-trivial $L$-space knot one has $\nu^+(K)=g(K)>0$ \cite{Hom2016refinement} and $\nu^+(mK)=0$ \cite[Lemma~16]{Gainullin2017mapping}.
\end{rem}


\section{Proof of Theorem~\ref{thm:12n242}}
Throughout this section we take $K=12n242$. Suppose that $p/q\neq 0$ is a non-characterizing slope for $K$ satisfying
 \begin{equation}\label{eq:length_hypothesis}
\ell_{K}(p/q)\geq 14.17> \frac{2\pi}{\sqrt{1-\left(\frac{\vol(4_1)}{\vol(12n242)}\right)^{\frac23}}}.
\end{equation} 
Let $K'\neq K$ be a knot in $S^3$ such that $S_{p/q}^3(K)\cong S_{p/q}^3(K')$. 

The length bound \eqref{eq:length_hypothesis} implies that the manifold $S_{p/q}^3(K)$ is hyperbolic and, using Lemma~\ref{lem:vol_length_bound}, that $S_{p/q}^3(K)$ cannot be obtained by any surgery on the figure-eight knot $4_1$.
By Thurston's trichotomy for knots in $S^3$, the knot $K'$ is either a torus knot, a hyperbolic knot or a satellite knot. Since torus knots never yield a hyperbolic manifold by surgery, we may ignore the first possibility and restrict our attention to the latter two options.

\begin{claim}\label{cl:hyp_bound}
If $K'$ is a hyperbolic knot, then 
\begin{equation*}\label{eq:hyperbolic_bound}
q<49 \quad  \text{and} \quad |p|< 49(2g(K')-1).
\end{equation*}
\end{claim}
\begin{proof}
Suppose that $K'$ is a hyperbolic knot. Condition \eqref{eq:length_hypothesis} eliminates the possibility that $K'$ is $4_1$. By consideration of the Casson-Walker invariant as in Lemma~\ref{lem:CW_obstruction}, we see that $K'$ is not $5_2$ or $m5_2$. Using the $\nu^+$ invariant as in Lemma~\ref{lem:nu+application}, we see that $K'$ is not $m12n242$. Thus having exhausted all the low volume knots in Theorem~\ref{thm:low_vol_knots}, we may conclude that $\vol(K')>3.07$. Thus by Lemma~\ref{lem:vol_length_bound} we have the bound
\begin{equation*}\label{eq:hyp_length_bound}
\ell_{K'}(p/q)<\frac{2\pi}{\sqrt{1-\left(\frac{\vol(12n242)}{3.07}\right)^{\frac23}}}<27.34.
\end{equation*}
Using Lemma~\ref{lem:bounding_pq}, this yields the required bound.
\end{proof}
\begin{claim}\label{cl:sat_q_geq_2}
If $K'$ is a satellite knot and $q\geq 2$, then
\begin{equation*}\label{eq:satellite_bound_q_geq2}
q<49 \quad  \text{and} \quad |p|<  49(2g(K')-1).
\end{equation*}
\end{claim}
\begin{proof} Suppose that $K'$ is a satellite knot and that $q\geq 2$. By \eqref{eq:length_hypothesis}, the manifold $S_{p/q}^3(K)$ is hyperbolic and Lemma~\ref{lem:satellite} applies to show that $K'$ is a cable of a hyperbolic knot $J$ such that $g(J)<g(K')$ and $S_{p/q'}^3(J)\cong S_{p/q}^3(K)$ for some $q'>q$. By the assumption \eqref{eq:length_hypothesis} we see that $J$ is not $4_1$. Furthermore, applying the Casson-Walker invariant as in Lemma~\ref{lem:CW_satellite_obstruction}, we see that $J$ cannot be $5_2$, $m5_2$, $12n242$ or $m12n242$. This is because there are no non-trivial integer solutions with $s\geq 2$ to the equations:
 \[
 24=\frac{(r^2-1)(s^2-1)}{12}+4s^2
 \] 
 and
  \[
 24=\frac{(r^2-1)(s^2-1)}{12}+24s^2.
 \]
Thus, having ruled out all the knots of small volume in Theorem~\ref{thm:low_vol_knots}, the only remaining possibility is that $J$ must be a knot with $\vol(J)> 3.07$. Thus by Lemma~\ref{lem:vol_length_bound} we have the bound
\begin{equation*}\label{eq:sat_length_bound1}
\ell_{J}(p/q')<\frac{2\pi}{\sqrt{1-\left(\frac{\vol(12n242)}{3.07}\right)^{\frac23}}}<27.34.
\end{equation*}
Applying Lemma~\ref{lem:bounding_pq} along with the inequalities $q<q'$ and $g(J)<g(K')$ give the required bounds.
\end{proof}
\begin{claim}\label{cl:sat_pq_geq_9}
If $K'$ is a satellite knot and $p/q\geq 9$, then
\begin{equation*}\label{eq:satellite_bound_lspace}
|p|<49(2g(K')-1),
\end{equation*}
\end{claim}
\begin{proof}
Suppose that $K'$ is a satellite knot and $p/q\geq 2g(K)-1=9$. Since $K$ is an $L$-space knot, this implies that $S_{p/q}^3(K)$ is a hyperbolic $L$-space. By Lemma~\ref{lem:satellite} there is a hyperbolic knot $J$ such that $g(J)<g(K')$ and $S_{p/q'}^3(J)\cong S_{p/q}^3(K)$ for some $q'>q$. Since $\Delta_{K}''(1)\neq 0$, \cite[Proposition~5.1]{Boyer1990Surgery} shows that $J$ is not $12n242$. Furthermore, since $S_{p/q'}^3(J)$ is an $L$-space and none of $4_1, 5_2, m5_2$ or $m12n242$ are $L$-space knots, Theorem~\ref{thm:low_vol_knots} allows us to conclude that $\vol(J)>3.07$. Thus as before we arrive at the bounds
\begin{equation*}\label{eq:sat_length_bound2}
\ell_{J}(p/q')<\frac{2\pi}{\sqrt{1-\left(\frac{\vol(12n242)}{3.07}\right)^{\frac23}}}<27.34
\end{equation*}
Applying Lemma~\ref{lem:bounding_pq}\eqref{it:p_bound} and $g(J)<g(K')$ gives the required bounds.
\end{proof}
We now convert these statements into results on characterizing slopes. The bound $q\geq 49$ is straight forward.
\begin{claim}
The slope $p/q$ is a characterizing slope for $K$ whenever $q\geq 49$.
\end{claim}
\begin{proof}
Together Claim~\ref{cl:hyp_bound} and Claim~\ref{cl:sat_q_geq_2} show that $p/q$ is a charactering slope for $K$ whenever $\ell_{K}(p/q)\geq 14.17$ and $q\geq 49$. However, Lemma~\ref{lem:bounding_pq}\eqref{it:q_bound} shows that $\ell_{K}(p/q)\geq 14.17$ is automatically satisfied whenever $q\geq 49$.
\end{proof}
In order to obtain the other conditions on charactering slopes, we need to invoke results linking the genera of $K$ and $K'$
\begin{claim}
The slope $p/q$ is a characterizing slope for $K$ whenever $p\geq \max\{24q, 441\}$.
\end{claim}
\begin{proof}
Since $S_{18}^3(K)$ is a lens space, it bounds a sharp 4-manifold. Thus \cite[Theorem~1.2]{McCoy2021sharp} applies to show that $S_{p/q}^3(K)$ bounds a sharp 4-manifold for all $p/q\geq 18$. In particular, we may apply \cite[Theorem~1.1]{McCoy2021sharp} to show that if $p/q\geq 4g(K)+4=24$, then $g(K')=g(K)=5$. Thus Claim~\ref{cl:hyp_bound} and Claim~\ref{cl:sat_pq_geq_9} imply that $p/q$ is a characterizing slope for $K$ whenever the conditions $p\geq 24q$, $p\geq 49(2g(K)-1)=441$ and $\ell_{K}(p/q)\geq 14.17$ are all satisfied. Lemma~\ref{lem:bounding_pq}\eqref{it:p_bound} shows that the bound $\ell_{K}(p/q)\geq 14.17$ is redundant, being implied by $p\geq 441$. Thus we have a characterizing slope for $K$ if $p\geq 24q$ and $p\geq 441$.
\end{proof}
\begin{claim}
The slope $p/q$ is a characterizing slope for $K$ whenever
\[q\geq 2\quad  \text{and}\quad p\leq - \max\{12+4q^2-2q,441 \}.\]
\end{claim}
\begin{proof}
By \cite[Theorem~1.8(ii)]{McCoy2020torus_char} we see that if $q\geq 2$ and $p\leq \min\{2q-12-4q^2, -10q\}$, then $g(K')=g(K)=5$. Thus Claim~\ref{cl:hyp_bound} and Claim~\ref{cl:sat_q_geq_2} imply that that $p/q$ is a characterizing slope for $K$ if the conditions $q\geq 2$, $p\leq -\max\{12+4q^2-2q, 10q\}$, $p\leq -441$ and $\ell_{K}(p/q)\geq 14.17$ are all satisfied. Since $12+4q^2-2q> 10q$ for all $q$ and the condition $p\leq -441$ implies $\ell_{K}(p/q)\geq 14.17$, we see that the conditions $p\leq -12+4q^2-2q$, $q\geq 2$ and $p\leq -441$ are sufficient to imply that $p/q$ is a characterizing slope for $K$.
\end{proof}
This completes the proof of all bounds in Theorem~\ref{thm:12n242}.

\section{Constructing some non-characterizing slopes}\label{sec:construction}
In this section we construct some examples of knots with non-characterizing slopes with arbitrarily large denominator. The generic construction is the following. Let $L=C'\cup K'$ be a link with two unknotted components and linking number $\lk(C',K')=\omega$. Let $Y$ be the manifold obtained by performing $1/n$-surgery on both components on $L$ for some non-zero integer $n\in \Z$. Since $C'$ and $K'$ are both unknotted, performing $1/n$ surgery on one or other of them individually again results in $S^3$. Performing such a surgery shows that $Y$ arises by $(n\omega^2+ \frac{1}{n})$-surgery on the knots $K$ and $C$, where $K$ is the image of $K'$ in the copy of $S^3$ obtained by surgering $C'$ and $C$ is the image of $C'$ after surgering $K'$. If one chose $L$ wisely, then the knots $K$ and $C$ will be distinct and thus the slope $n\omega^2+ \frac{1}{n}$ will be non-characterizing for $K$ and $C$.

Using this idea, we can prove the following.
\begin{prop}\label{prop:non-char}
Let $K$ be a knot with $g(K)\geq 2$ which can be unknotted by adding $q$ positive full twists along two oppositely oriented strands. Then $\frac{1}{q}$ is a non-characterizing slope for $K$.
\end{prop}
\begin{proof}
The hypothesis on unknotting implies that we can take a link $L=C'\cup K'$ with unknotted components such that (a) $K$ can be obtained from $K'$ by performing $1/q$-surgery on $C'$ and (b) $C'$ bounds a disk $D$ that intersects $K'$ in two oppositely oriented points. If we take the disk $D$ and add a tube that follows an arc of $K'$, we obtain an embedded genus one surface $\Sigma$ with boundary $C'$ which is disjoint from $K'$. Since $\Sigma$ is disjoint from $K'$, it is preserved under surgery on $K'$ and hence shows that the knot $C$ obtained by performing $1/q$ surgery on $K'$ has genus at most one. Since $K$ is assumed to have genus at least two, this implies that $C$ is not isotopic to $K$ and hence that $1/q$ is a non-characterizing slope for $K$.
\end{proof}
\begin{exam} Using the preceding proposition, we can show that for every $q\geq 1$, there is a hyperbolic 2-bridge knot $K_q$ for which $\frac{1}{q}$ is a non-characterizing slope. Figure~\ref{fig:2bridge_exam} depicts a two-bridge knot $K_q$ of genus two that can be unknotted by adding $q$ positive full twists along two oppositely-oriented strands. The genera of these knots can be easily verified, since Seifert's algorithm always yields a minimal genus Seifert surface when applied to an alternating diagram \cite{Crowell1959alternating, Murasugi1958alternating}. Thus Proposition~\ref{prop:non-char} applies to $K_q$.
\end{exam}
We also note that sufficiently complicated knots with unknotting number one must always have an non-characterizing slope. Since every slope is characterizing for the trefoil and the figure-eight knot \cite{Ozsvath2006trefoilsurgery}, we see that the condition on the genus cannot be relaxed.
\begin{cor}
Let $K$ be a knot with $g(K)\geq 2$ and $u(K)=1$.
\begin{itemize}
\item If $K$ can be unknotted by changing a positive crossing, then $+1$ is non-characterizing for $K$.
\item  If $K$ can be unknotted by changing a negative crossing, then $-1$ is non-characterizing for $K$.
\end{itemize} 
\end{cor}

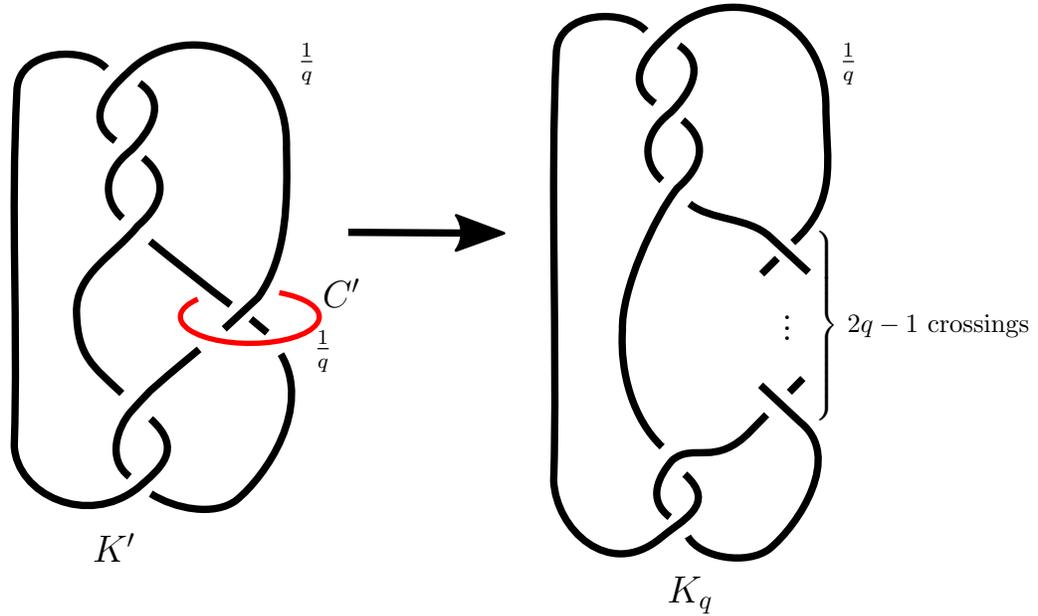
\begin{figure}
  \centerline{
    \begin{overpic}[width=0.7\textwidth]{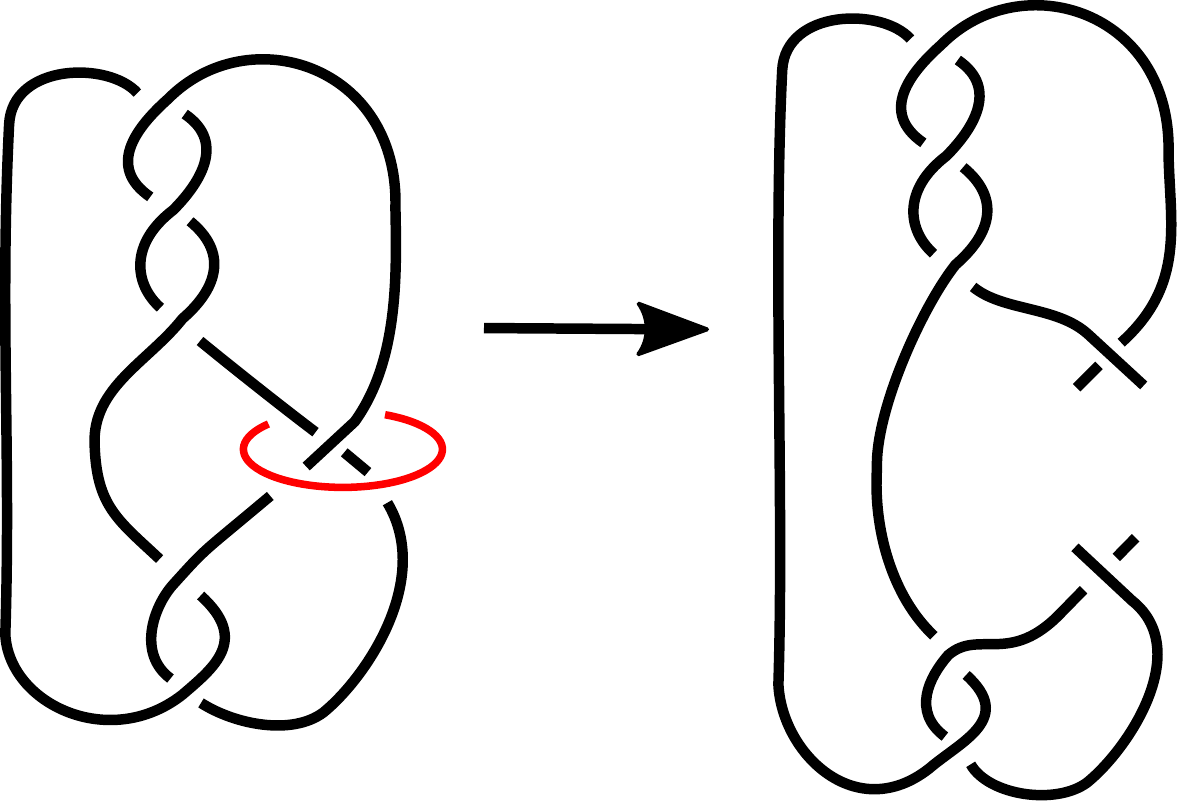}
      \put (10,0) {\LARGE $K'$}
      \put (38,31) {\LARGE $C'$}
      \put (80,-5) {\LARGE $K_q$}
      \put (37,25) {\Large $\frac{1}{q}$}
      \put (35,60) {\Large $\frac{1}{q}$}
      \put (101,60) {\Large $\frac{1}{q}$}
      \put (98,28){$\vbigl. \vbigr\}$ $2q-1$ crossings}
      \put (94,27){\LARGE $\vdots$}
    \end{overpic}
  }
  \vspace{1cm}
\caption{A link $K'\cup C'$, such that twisting along $C'$ yields the two-bridge link $K_q$. Proposition~\ref{prop:non-char} implies that $K_q$ has $\frac{1}{q}$ as a non-characterizing slope.}
\label{fig:2bridge_exam}
\end{figure}

\bibliographystyle{alpha}
\bibliography{master}
\appendix
\section{Calculating $\Delta_{T_{r,s}}''(1)$}
We conclude with a derivation of \eqref{eq:torus_deriv}. It will be convenient to define, for any positive integer $k$, the function 
\[Q_k(t)=t^{\frac{1-k}{2}}\left(\frac{t^k-1}{t-1}\right)=t^{\frac{1-k}{2}}\left(\sum_{i=0}^{k-1} t^i\right).\]
Using these, we can write the Alexander polynomial of a torus knot in the form:
\[
\Delta_{T_{r,s}}(t)=\frac{Q_{rs}(t)}{Q_{r}(t)Q_{s}(t)}.
\]
Since $Q_k(t)=Q_{k}(t^{-1})$, we have that 
\[Q_k'(1)=0.\]
Furthermore, we calculate that
\[Q_k(1)=k\]
and
\begin{align*}
Q_k''(1)&=\sum_{i=0}^{k-1} \left(i - \frac{k-1}{2}\right)\left(i - \frac{k+1}{2}\right)= \sum_{i=0}^{k-1} \left(i^2 -ki + \frac{(k-1)(k+1)}{4}\right)\\
&=\frac{k(k-1)(2k-1)}{6}- \frac{k^2(k-1)}{2}+ \frac{k(k-1)(k+1)}{4}\\
&=\frac{k(k^2-1)}{12}.
\end{align*}
These identities allow us to calculate $\Delta_{T_{r,s}}''(1)$ implicitly. Differentiating the identity 
\[Q_{r}(t)Q_{s}(t)\Delta_{T_{r,s}}(t)=Q_{rs}(t)\]
twice and evaluating at $t=1$, we obtain
\begin{align*}
Q_{rs}''(1)&=\frac{rs(r^2s^2-1)}{12}\\
&= \left(Q_{r}(1)Q_{s}(1)\right)''\Delta_{T_{r,s}}(1)+2\left(Q_{r}(1)Q_{s}(1)\right)'\Delta_{T_{r,s}}'(1)+ Q_{r}(1)Q_{s}(1)\Delta_{T_{r,s}}''(1)\\
&=Q_{r}''(1)Q_{s}(1)+2Q_{r}'(1)Q_{s}'(1)+ Q_{r}(1)Q_{s}''(1)+rs\Delta_{T_{r,s}}''(1)\\
&=\frac{rs(r^2-1)}{12}+\frac{rs(s^2-1)}{12}+rs\Delta_{T_{r,s}}''(1).
\end{align*}
From this one rearranges to obtain the desired formula:
\[
\Delta_{T_{r,s}}''(1)=\frac{(r^2-1)(s^2-1)}{12}.
\]
\end{document}